\documentclass[11pt]{article}
\pdfoutput=1
\usepackage[margin=1in]{geometry}

\DeclareMathAlphabet{\mathbbold}{U}{bbold}{m}{n}
\usepackage{dsfont}
\usepackage[usenames,dvipsnames,svgnames,table]{xcolor}

\usepackage[utf8]{inputenc} 
\usepackage[T1]{fontenc}    
\usepackage{lmodern}
\usepackage{url}            
\usepackage{booktabs}       
\usepackage{amsfonts}       
\usepackage{nicefrac}       
\usepackage{microtype}      
\usepackage{amsmath,amsthm,amssymb,mathrsfs}
\usepackage{mathtools}
\usepackage{thm-restate}
\usepackage[usenames,dvipsnames,svgnames,table]{xcolor}
\usepackage{braket}
\usepackage{enumitem}

\newtheorem{theorem}{Theorem}

\newtheorem{definition}{Definition}
\newtheorem{lemma}{Lemma}
\newtheorem{corollary}{Corollary}

\usepackage[nobiblatex]{xurl} 
\usepackage[pagebackref]{hyperref}
\hypersetup{
	pdftitle={Near-Optimal Lower Bounds For Convex Optimization For All Orders of Smoothness}, 
	pdfauthor={Ankit Garg, Robin Kothari, Praneeth Netrapalli, Suhail Sherif}, 
	colorlinks=true, 
	linkcolor=blue, 
	citecolor=blue, 
	urlcolor=blue 
}

\renewcommand{\backref}[1]{}

\renewcommand{\backrefalt}[4]{%
	\ifcase #1 %
	\or
	[p.\ #2]%
	\else
	[pp.\ #2]%
	\fi}

\makeatletter
\newcommand{\para}{%
	\@startsection{paragraph}{4}%
	{\z@}{2ex \@plus 3.3ex \@minus .2ex}{-1em}%
	{\normalfont\normalsize\bfseries}%
}
\makeatother

\usepackage[nameinlink]{cleveref}

\newtheorem{proposition}[theorem]{Proposition}

\newcommand{\R}{\mathbb{R}}
\newcommand{\E}{\mathop{{}\mathbb{E}}}

\newcommand{\dotp}[2]{\langle #1, #2 \rangle}
\newcommand{\abs}[1]{\left| #1 \right|}

\newcommand{\iprod}[2]{\left\langle #1, #2 \right \rangle}
\newcommand{\defeq}{\coloneqq}

\newcommand{\mF}{\mathcal{F}}

\newcommand{\eps}{\epsilon}

\newcommand{\N}{\mathbb{N}}

\renewcommand{\S}{\mathbf{S}}

\newcommand{\tO}{\widetilde{O}}

\newcommand{\floor}[1]{\left\lfloor{#1}\right\rfloor}

\newcommand{\norm}[1]{\|{#1}\|}
\newcommand{\Norm}[1]{\left\|{#1}\right\|}

\renewcommand{\>}{\rangle}

\DeclareRobustCommand{\noop}[1]{}

\DeclareMathOperator{\poly}{poly}

\DeclareMathOperator*{\argmax}{argmax}

\renewcommand{\S}{\mathcal{S}}
\newcommand{\smax}{\mathsf{smax}}
\newcommand{\aff}{\mathsf{vec}}
\newcommand{\Oracle}{\mathsf{Q}}
\newcommand{\mR}{\mathcal{R}}

\begin{document}
	
	\title{Near-Optimal Lower Bounds For Convex Optimization\\ For All Orders of Smoothness}
	
	\author{
		Ankit Garg\footnote{Microsoft Research India. Email: \texttt{garga@microsoft.com}}
		\and
		Robin Kothari\footnote{Microsoft Quantum and Microsoft Research. \texttt{robin.kothari@microsoft.com}}
		\and 
		Praneeth Netrapalli\footnote{Microsoft Research India. Part of the work was done after the author moved to Google Research India. Email: \texttt{pnetrapalli@google.com}}
		\and
		Suhail Sherif\footnote{Vector Institute, Toronto, ON, Canada. Email: \texttt{suhail.sherif@gmail.com}}
	}
	\date{}
	
	\maketitle
	
	\begin{abstract}
		We study the complexity of optimizing highly smooth convex functions. 
		For a positive integer $p$, we want to find an $\eps$-approximate minimum of a convex function $f$, given oracle access to the function and its first $p$ derivatives, assuming that the $p$th derivative of $f$ is Lipschitz. 
		
		Recently, three independent research groups (Jiang et al., PLMR 2019; Gasnikov et al., PLMR 2019; Bubeck et al., PLMR 2019) developed a new algorithm that solves this problem with $\tO(1/\eps^{\frac{2}{3p+1}})$ oracle calls for constant $p$. This is known to be optimal (up to log factors) for deterministic algorithms, but known lower bounds for randomized algorithms do not match this bound. We prove a new lower bound that matches this bound (up to log factors), and holds not only for randomized algorithms, but also quantum algorithms.
	\end{abstract}
	
	\section{Introduction}
\label{sec:intro}

In recent years, several optimization algorithms have been proposed, especially for machine learning problems, that achieve improved performance by exploiting the smoothness of the function to be optimized. 
Specifically, these algorithms have better performance when the function's first, second, or higher order derivatives are Lipschitz~\cite{Nes08,baes2009estimate,MS13,nesterov2019implementable,GDG+19,JWZ19,bubeck19ahighlysmooth}.

In this paper we study the problem of minimizing a highly smooth convex function given black-box access to the function and its higher order derivatives. The simplest example of the family of problems we consider here is the problem of approximately minimizing a convex function $f:\R^n \to \R$, given access to an oracle that on input $x \in \R^n$ outputs $(f(x),\nabla f(x))$, under the assumption that the function's first derivative, its gradient $\nabla f$, has bounded Lipschitz constant. This problem can be solved by Nesterov's accelerated gradient descent, and it is known that this algorithm is optimal (in high dimension) among deterministic and randomized algorithms \cite{Nes83,nemirovsky1983problem}. 

More generally, for any positive integer $p$, consider the $p$th-order optimization problem: For known $R>0$, $\eps>0$, and $L_p>0$, we have a $p$ times differentiable convex function $f:\R^n \to \R$ whose $p$th derivative has Lipschitz constant at most $L_p$, which means
\begin{equation}\label{eq:Lp}
    \norm{\nabla^p f(x) - \nabla^p f(y)} \leq L_p \norm{x-y},
\end{equation}
where $\norm{\cdot}$ is the $\ell_2$ norm (for vectors) or induced $\ell_2$ norm 
(for operators)\footnote{For a symmetric $p^{\textrm{th}}$ order tensor $T$, the induced norm is defined as $\max_{x_i \in \R^n} \frac{{{T(x_1,x_2,\cdots,x_p)}}}{\prod_{i=1}^p \norm{x_i}}$.}. 
Our goal is to find an $\eps$-approximate minimum of this function in a ball of radius $R$, which is any $x^*$ that satisfies
\begin{equation}\label{eq:eps}
    f(x^*)-\min_{x\in B_R(0)} f(x) \leq \eps, 
\end{equation}
where $B_R(0)$ is the $\ell_2$-ball of radius $R$ around the origin. We can access the function $f$ through a $p$th order oracle, which when queried with a point $x \in \R^n$ outputs 
\begin{equation}\label{eq:oracle}
    (f(x), \nabla f(x), \ldots, \nabla^p f(x)).
\end{equation}
As usual, $\nabla^p f(x)$ denotes the $p$th derivative of $f(x)$.

Our primary object of study will be the minimum query cost of an algorithm that solves the problem, i.e. the number of queries (or calls) to the oracle in \cref{eq:oracle} that an algorithm has to make.\footnote{For simplicity we assume that the oracle's output is computed to arbitrarily many bits of precision. 
This only makes our results stronger, since we prove lower bounds in this paper.
} 
For a fixed $p$, it seems like this problem has 4 independent parameters, $n$, $L_p$, $R$, and $\eps$, but the parameters are not all independent since we can scale the input and output spaces of the function to affect the latter 3 parameters. Thus the complexity of any algorithm can be written as a function of $n$ and $L_pR^{p+1}/\eps$. In this paper we focus on the high-dimensional setting where $n$ may be much larger than the other parameters, and the best algorithms in this regime have complexity that only depends on $L_pR^{p+1}/\eps$ with no dependence on $n$.

As noted, the $p=1$ problem has been studied since the early 80s \cite{Nes83,nemirovsky1983problem}, and the $p>1$ problem has also been considered~\cite{Nes08,MS13}. 
In particular, it was known \cite{baes2009estimate,nesterov2019implementable} that the problem can be solved by a deterministic algorithm that makes
\begin{equation}
    \tilde{O}_p\left(\left({L_pR^{p+1}}/{\epsilon}\right)^{1/(p+1)} 
    \right)
\end{equation}
oracle calls,\footnote{Note that the query complexity does not have any dependence on the dimension $n$. Of course, actually implementing each query will take poly$(n)$ time, but we only count the number of queries here.} where the subscript $p$ in the big Oh (or big Omega) notation means the constant in the big Oh can depend on $p$. In other words, this notation means that we treat $p$ as a constant.

In an exciting recent development, new algorithms were proposed for all $p$ (with very similar complexity) by three independent groups of researchers: Gasnikov, Dvurechensky, Gorbunov, Vorontsova, Selikhanovych, and Uribe~\cite{GDG+19}; Jiang, Wang, and Zhang~\cite{JWZ19}; Bubeck, Jiang, Lee, Li and Sidford~\cite{bubeck19ahighlysmooth}. 
All three groups develop deterministic algorithms that make
\begin{equation}
    \tilde{O}_p\left(\left({L_pR^{p+1}}/{\epsilon}\right)^{2/(3p+1)} 
    \right)
\end{equation}
oracle calls. 

This algorithm is nearly optimal among deterministic algorithms, since the works \cite{nesterov2019implementable,ArjevaniSS19} showed that any deterministic algorithm that solves this problem must make $\Omega_p\left(\left({L_pR^{p+1}}/{\epsilon}\right)^{2/(3p+1)}\right)$
queries. 
However, for randomized algorithms, the known lower bound is weaker. Agarwal and Hazan~\cite{agarwal2018lower} showed that any randomized algorithm must make  
\begin{equation}
    \Omega_p\left(\left({L_pR^{p+1}}/{\epsilon}\right)^{2/(5p+1)}\right)    
\end{equation}
queries. To the best of our knowledge, no lower bounds are known in the setting of high-dimensional smooth convex optimization against quantum algorithms, although quantum lower bounds are known in the low-dimensional setting~\cite{CCLW20,vAGGdW20} and for non-smooth convex optimization~\cite{GKNS21}.

In this work, we close the gap (up to log factors) between the known algorithm and randomized lower bound for all $p$. Furthermore, our lower bound also holds against quantum algorithms.

\begin{theorem}\label{thm:main}
    Fix any $p \in \N$. For all $\epsilon>0$, $R>0$, $L_p>0$, there exists an $n>0$ and a set of $n$-dimensional functions $\mF$ with $p$th-order Lipschitz constant $L_p$ (i.e., satisfying \cref{eq:Lp}) such that any randomized or quantum algorithm that outputs an $\eps$-approximate minimum (satisfying \cref{eq:eps}) for any function $f \in \mF$ must make 
    \begin{equation}
        \Omega_p\left(\left( {L_pR^{p+1}}/{\epsilon} \right)^{2/(3p+1)} \left(\log{L_pR^{p+1}}/{\epsilon}\right)^{-2/3}\right)    
    \end{equation}
    queries to a $p$th order oracle for $f$ (as in \cref{eq:oracle}).
\end{theorem}

In fact, this lower bound holds even against highly parallel randomized algorithms, where the algorithm can make poly($n,L_pR^{p+1}/{\epsilon}$) queries in each round and we only count the total number of query rounds (and not the total number of queries). See~\cite{BJLLS19} for previous work in this setting, including speedups for first-order convex optimization in the low dimensional setting. 

In this introduction, we have deliberately avoided explaining the quantum model of computation to make the results accessible to readers without a background in quantum computing. The entire paper is written so that the randomized lower bound is fully accessible to any reader who does not wish to understand the quantum model and quantum lower bound. 
For readers familiar with quantum computing, we note that the only thing to be changed to get the quantum model is to modify the oracle in \cref{eq:oracle} to support queries in quantum superposition. This is done in the usual way, by defining a unitary implementation of the oracle, which allows quantum algorithms to make superposition queries and potentially solve the problem more efficiently than randomized algorithms.

 	\section{High level overview}
Let us first consider the lower bound against randomized algorithms. Let us also first look at the special setting of $p=0$ where we still assume access to the gradient (or $p=1$) oracle. To be more precise, the oracle returns subgradients, since gradients need not be defined at all points for Lipschitz convex functions.
For this setting, known popularly as \emph{nonsmooth} convex optimization, the optimal lower bound of ${\Omega}\left(\frac{1}{\epsilon^2}\right)$ is in fact a classical result~\cite{nemirovsky1983problem}. 
The proof of this result is very elegant and has been used subsequently to prove several other related lower bounds such as for parallel randomized algorithms~\cite{nemirovski1994parallel}, quantum algorithms~\cite{GKNS21} etc. Since our proof builds on this framework, we now review this.

\para{Nonsmooth lower bound instance.} The lower bound instance for nonsmooth convex optimization is $\min_{\norm{x}\leq 1} f_V(x)$ where $f_V:\R^n \to \R$ is chosen as
\begin{align}\label{eqn:nonsmooth}
	f_V(x) = \max_{i \in [k]} \iprod{v_i}{x} + (k-i) \gamma,
\end{align}
with $k= O\left(n^{1/3}\right)$, $\gamma = \tilde{\Theta}\left(\frac{1}{\sqrt{n}}\right)$, and $V=\left(v_1,\cdots,v_k\right)$ comprises  $k$ orthonormal vectors chosen uniformly at random. The argument essentially shows that
\begin{enumerate}[label=(\roman*)]
	\item in order to find a $O\left(\frac{1}{\sqrt{k}}\right)$-approximate minimizer, one needs to know all the $v_i$'s, and
	\item with high probability, each query reveals at most one new vector $v_i$.
\end{enumerate}
This yields a lower bound of $k$ queries for achieving an error of $O\left(\frac{1}{\sqrt{k}}\right)$. Since $f$ is a $1$-Lipschitz function, when we rewrite this bound in terms of error, it yields the $\Omega\left(\frac{1}{\epsilon^2}\right)$ randomized lower bound.

For the $p=1$ setting, known popularly as \emph{smooth} convex optimization, the optimal lower bound $\Omega\left(\sqrt{\frac{1}{\epsilon}}\right)$ is also a classical result originally proven in~\cite{nemirovsky1983problem}. However, the proof of this result in~\cite{nemirovsky1983problem} is quite complicated and is not widely known. The recent papers of~\cite{guzman2015lower,diakonikolas2020lower} provide a much simpler proof of the $p=1$ result by using the lower bound construction for $p=0$ setting described above and using \emph{smoothing}, which we now review.

\para{Smoothing.} Smoothing refers to the process of approximating a given Lipschitz function $f$ by another function $g$, which has Lipschitz continuous $p^\textrm{th}$ derivatives for some $p \geq 1$. Further, since we will be applying this operation to~\eqref{eqn:nonsmooth}, we will describe smoothing in this context.
\begin{definition}\label{def:smoothing}
	An operator $\S$ which takes a $1$-Lipschitz function $f$ 
	to another convex function $\S[f]$ is called a $(p,\beta,\eps)$-smoothing operation if it satisfies the following:
	\begin{enumerate}
		\item	\emph{Smoothness}: $p^\textrm{th}$ derivatives of $\S[f]$ are Lipschitz continuous with parameter $\beta$, and
		\item	\emph{Approximation}: For any $x$, we have $\abs{f(x)-\S[f](x)} \leq \eps$.
	\end{enumerate}
\end{definition}
If we can design a smoothing operation as per the above definition with $\eps=O\left(1/\sqrt{k}\right)$ and further ensure that property (ii) above i.e., \emph{with high probability, each query to the first $p$ derivatives of $\S[f]$ reveals at most one new vector $v_i$}, then the proof strategy of lower bound for nonsmooth convex optimization can be executed on the smoothed instance $\S[f]$, there by giving us a lower bound for $p^{\textrm{th}}$-order smooth convex optimization. This is the key idea of~\cite{guzman2015lower,diakonikolas2020lower}. Further, the smaller $\beta$ is, the better the bound we obtain.
However, since $f$ can have discontinuous $p^{\textrm{th}}$ derivatives, there is a tension between the approximation property which tries to keep $\S[f]$ close to $f$ and the smoothness property. So, one cannot make $\beta$ very small after fixing $\eps = O(1/\sqrt{k})$. For the rest of this section, we fix $\eps = O(1/\sqrt{k})$ in \Cref{def:smoothing}.

For the $p=1$ setting, there is a well-known smoothing operation known as \emph{Moreau/inf-conv} smoothing~\cite{Bauschke2011}, which obtains the best possible smoothing with $\beta = \Theta(k^{1.5})$. This gives the tight query lower bound of $\Omega\left({\frac{1}{\sqrt{\epsilon}}}\right)$ for smooth convex optimization. 

However, there is no known generalization of inf-conv smoothing for $p \geq 2$, so one needs to use a different smoothing operator to extend this proof strategy for proving query lower bounds for higher order smooth convex optimization. Given any $p\geq 1$,~\cite{agarwal2018lower} indeed construct such a smoothing, called \emph{randomized smoothing} which maps Lipschitz convex functions to convex functions with Lipschitz $p^{\textrm{th}}$ derivatives. In the general $p \geq 1$ setting, a smoothing operator with $\beta = O\left(k^{3p/2}\right)$ would 
give the optimal lower bound of $\Omega\Bigl({\epsilon}^{\frac{-2}{3p+1}}\Bigr)$. However, the randomized smoothing of~\cite{agarwal2018lower} can obtain only $\beta = O\left(k^{5p/2}\right)$ leading to a suboptimal $\Omega\Bigl(\eps^{\frac{-2}{5p+1}}\Bigr)$ lower bound for $p^{\textrm{th}}$ order smooth convex optimization.

We design an improved smoothing operation for the specific class of functions in \cref{eqn:nonsmooth}, with the optimal $\beta = O\left(k^{3p/2}\right)$ using two key ideas. The first idea is the \emph{softmax} function with parameter $\rho$ defined as $\smax_{\rho}(z)\defeq \rho \log\left(\sum_{i \in [k]} \exp\left(\frac{z_i}{\rho}\right)\right)$, where $z \in \R^k$. If we apply $\smax_{\rho}$, with $\rho = k^{-3/2}$, to functions of the form~\eqref{eqn:nonsmooth} through:
\begin{align}
	\smax_\rho(\aff_V(x)) \defeq \rho \log\left(\sum_{i\in[k]} \exp\left(\frac{\iprod{v_i}{x} + (k-i)\gamma}{\rho}\right)\right),
\end{align}
where $\aff_V(x) \defeq (\dotp{v_1}{x} + (k-1)\gamma,\dotp{v_2}{x} + (k-2)\gamma, \dots, \dotp{v_k}{x})$, we can show that $\smax_{\rho}(\aff_V(x))$
satisfies \Cref{def:smoothing} with the optimal value of $\beta = {O}\left(k^{3p/2}\right)$. However, any query on derivatives of $\smax_\rho(\aff_V(x))$ reveals information about all the vectors $v_i$  simultaneously since for instance the gradient is given by
\begin{equation}
	\nabla \smax_{\rho}(\aff_V(x)) = \sum_{i\in[k]} \frac{\exp\left(\frac{\iprod{v_i}{x} + (k-i)\gamma}{\rho}\right)}{ \sum_{j \in [k]} \exp\left(\frac{\iprod{v_j}{x} + (k-j)\gamma}{\rho}\right)} \cdot v_i.	
\end{equation}
Consequently, it cannot be directly used to obtain a lower bound.
The second idea is that even though the function value and derivatives of $\smax_{\rho}(\aff_V(x))$ have contribution from all $v_i$'s, the contribution is heavily dominated (i.e., up to $\frac{1}{\poly(k)}$ error) by $v_{i^*(x)}$, where $i^*(x) = \argmax_{i \in [k]} \iprod{v_i}{x} + (k-i)\gamma$, whenever $\iprod{v_{i^*(x)}}{x} + (k-{i^*(x)})\gamma > \iprod{v_i}{x} + (k-i)\gamma + \Omega(\rho \log k)$ for every $i \neq i^*(x)$.

Based on this insight, we design a new $1$-Lipschitz convex function given by $h(x) \defeq \max_{i \in [k]} f_i(x)$ where $f_i(x) \defeq \smax^{\leq i}_{\rho}(\aff_V(x)) + \rho(k-i)n^{-\alpha}$ for an appropriate $\alpha$ to be chosen later, where $\smax^{\leq i}_{\rho}(\aff_V(x)) \defeq \rho \log\left(\sum_{j\in[i]} \exp\left(\frac{\iprod{v_j}{x} + (k-j)\gamma}{\rho}\right)\right)$. The key property satisfied by $f_i$ is that $f_i \approx f_j$ implies that $\nabla f_i \approx \nabla f_j$ for any $i,j$. This implies that near points of discontinuous gradients for $h$ i.e., points where $\argmax_{i \in [k]} f_i(x)$ changes, the resulting discontinuity in $\nabla h(x)$ is $O\left(\frac{1}{\poly(k)}\right)$. In contrast, the change in gradients of the original instance $f_V(x)$ near points of discontinuity is $\Omega(1)$. If we apply randomized smoothing to $h$, the resulting function can then be shown to have $p^\textrm{th}$ order Lipschitz constant $\widetilde{O}\left(k^{3p/2}\right)$. The precise details, proved in \Cref{lem:gsmoothness}, are technical and form the bulk of this paper.
The same proof strategy immediately yields the same bound on the number of \emph{rounds} for \emph{parallel} randomized algorithms as long as the number of queries in each round is at most $\poly(k)$. The reason is that $\poly(k)$ queries are still not sufficient to obtain information about more than one vector per round. Finally, the same proof strategy can be adapted to the quantum setting using the \emph{hybrid argument}~\cite{BBBV97}. See \Cref{sec:infhiding} for more details.
 	\section{Smoothing preliminaries}\label{sec:smoothing}

In this section we look at some smoothing functions and their properties. The proofs of these properties can be found in \Cref{sec:prelims}.

Let $B_{\eta}(x)=\{y \in \R^n: \norm{y-x}\leq \eta\}$ be the ball of radius $\eta$ around $x \in \R^n$.

\begin{definition}[Randomized smoothing]
    For any function $f:\R^n \rightarrow \R$ and real-valued $\eta > 0$, the randomized smoothing operator $S_\eta$ produces a new function $S_{\eta}[f]:\R^n \rightarrow \R$ from $f$ with the same domain and range, defined as 
    \begin{equation}
        S_{\eta}[f](x) = \E_{y \in B_{\eta}(x)}[f(y)].    
    \end{equation}      
\end{definition}

This smoothing turns non-smooth functions into smooth functions. If we start with a function $f$ that is Lipschitz, then after randomized smoothing, the resulting function's first derivative will be defined and Lipschitz~\cite{agarwal2018lower}. Since we want to construct functions with $p$ derivatives, we define a $p$-fold version of randomized smoothing.
Recall that $p$ is the same $p$ as in the introduction (i.e., we are proving lower bounds on the $p$th order optimization problem). This operation also depends on a parameter $\beta$ that we will fix later.

\begin{restatable}[Smoothing]{definition}{smoothing}
    The smoothing operator $\S$ on input $f:\R^n \to \R$ outputs the function 
    \begin{equation}
        \S[f] = S_{\beta/2^{p}}[S_{\beta/2^{p-1}}[\cdots S_{\beta/2^{2}}[S_{\beta/2^{1}}[f]]\cdots]].
    \end{equation}
\end{restatable}

The main properties we require from this smoothing are as follows.

\begin{restatable}{lemma}{smoothingproperties}
    \label{lem:smoothing}
    For the smoothing operator $\S$ defined above, the following statements hold true.
    \begin{enumerate}
        \item For any functions $f,g : \R^n \rightarrow \R$ for which $\S[f], \S[g]$ are well-defined, $\S[f+g] = \S[f] + \S[g]$.
        \item The value $\S[f](x)$ only depends on the values of $f$ within a $(1-2^{-p})\beta$ radius of $x$.
        \item The gradient and higher order derivatives of $\S[f]$ at $x$ depend only on the values of $f$ within $B_{\beta}(x)$.
        \item If $\nabla^p f$ is $L$-Lipschitz in a ball of radius $\beta$ around $x$, then $\nabla^p \S[f]$ is also $L$-Lipschitz at $x$.
        \item Let $f$ be $G$-Lipschitz in a ball of radius $\beta$ around $x$. Then $\S[f]$ is $p$-times differentiable, and for any $i \leq p$, $\nabla^i \S[f]$ is $L$-Lipschitz in a $\beta/2^p$-ball around $x$ with $L \leq \frac{n^i 2^{i(i+1)/2}}{\beta^i} G$.
        \item Let $f$ be $G$-Lipschitz in a ball of radius $\beta$ around $x$. Then $\abs{\S[f](x) - f(x)} \leq \beta G$.
        \item If $f$ is a convex function, then $\S[f]$ is also a convex function.
    \end{enumerate}
\end{restatable}

We also use the softmax function introduced earlier.

\begin{restatable}[Softmax]{definition}{softmax}
    For a real number $\rho$, the softmax function $\smax_{\rho}: \R^n \rightarrow \R$ is defined as \begin{equation}
        \smax_{\rho}(x) = \rho \ln\Bigl(\sum_{i \in [n]} \exp(x_i/\rho)\Bigr).  
    \end{equation}         
    Let us also define, for $m \leq n$, $\smax^{\leq m}_{\rho}: \R^n \rightarrow \R$ as 
    \begin{equation}
        \smax^{\leq m}_{\rho}(x) = \smax_{\rho}(x_{\leq m}) \text{, or equivalently, } \smax^{\leq m}_{\rho}(x) = \rho \ln\Bigl(\sum_{i \in [m]} \exp(x_i/\rho)\Bigr). 
    \end{equation}
\end{restatable}

We note the following smoothness properties of softmax.

\begin{restatable}{lemma}{softmaxproperties}
    \label{lem:smaxsmoothness}
    The following are true of the function $\smax_{\rho}$ for any $\rho>0$.
    \begin{enumerate}
        \item The first derivative of $\smax$ can be computed as \begin{equation} \frac{\partial \smax_{\rho}(x)}{\partial x_i} = \frac{\exp(x_i/\rho)}{\sum_i \exp(x_i/\rho)}. \end{equation}
        \item $\smax_{\rho}$ is $1$-Lipschitz and convex.
        \item The higher-order derivatives of $\smax$ satisfy \begin{equation} \| \nabla^p \smax_{\rho}(x) - \nabla^p \smax_{\rho}(y) \| \leq \frac{\left(\frac{p+1}{\ln(p+2)}\right)^{p+1} p!}{\rho^{p}} \|x-y\|. \end{equation}
    \end{enumerate}
\end{restatable}

We will also need the following lemma which roughly states that if $\smax(x)$ and $\smax^{\leq m}(x)$ are nearly the same, then their gradients are also nearly the same at $x$.

\begin{restatable}{lemma}{softmaxgradients}
    \label{lem:nearlysamesmax}
    Let $x \in \R^n$ and $m<n$. If \begin{equation} \frac{\smax_{\rho}(x) - \smax^{\leq m}_{\rho}(x)}{\rho} = \delta < 1, \end{equation} Then 
    \begin{equation} \| \nabla \smax_{\rho}(x) - \nabla \smax^{\leq m}_{\rho}(x) \| \leq 4\delta. \end{equation} 
\end{restatable}

\section{Function construction and properties}\label{sec:func-props}

In this section we define the class of functions used in our randomized (and quantum) lower bound, and state the properties of the function that will be exploited in the lower bound.

Let $k \in \N, \gamma, \beta, \rho, \alpha \in \R$ be parameters to be defined shortly. $k$ and $\gamma$ are parameters as used in the high level overview~(\cref{eqn:nonsmooth}), $\beta$ is the parameter required to define $\S$ and $\rho$ is the parameter used in the definition of $\smax$.

\para{Function construction.}
Given a list of orthonormal vectors $v_1, \dots, v_k \in \R^n$, which we collectively call $V$, we recall  that $\aff_V(x) \in \R^k$ denotes the vector 
\begin{equation}
    \aff_V(x) = (\dotp{v_1}{x} + (k-1)\gamma,\dotp{v_2}{x} + (k-2)\gamma, \dots, \dotp{v_k}{x}).
\end{equation}
We can now define our hard function class as follows.

\begin{definition}
    Let $V = (v_1, \dots, v_k) \in \R^n$ be a set of orthonormal vectors. The functions $f_1$, \dots, $f_k$, $h$ and $g$ depend on $V$ as follows.
    Define, for each $i \in [k]$, the function $f_i: \R^n \rightarrow \R$ as \begin{equation} f_{i}(x) = \smax^{\leq i}_{\rho}(\aff_V(x)) + \rho(k-i)n^{-\alpha}. \end{equation}
    Define $h(x) = \max_{i \in [k]} f_i(x)$, and $g(x) = \S[h](x)$.
\end{definition}

Note that in the above definition we apply the $\smax$ functions on $\aff_V(x)$ and not on $x$. However, since $\aff_V(x)$ is obtained by applying a unitary transform on $x$ and then translating it, the observations about $\smax(x)$ in~\Cref{lem:smaxsmoothness,lem:nearlysamesmax} also hold for $\smax(\aff_V(x))$.

We set $\gamma = 40\sqrt{\frac{\ln n}{n}}$, $k = \floor{(0.1/\gamma)^{2/3}}$ (or $\gamma \approx 0.1/k\sqrt{k}$), $\rho = \gamma/100 \alpha\ln n, \beta = \gamma/\ln n, \alpha = p+1$.

\para{Function properties.} We now state some properties of the function that will be used to show the lower bounds.

\begin{restatable}{lemma}{smoothnesslemma}
    \label{lem:gsmoothness}
    For any choice of $V$, the function $g$ is convex, $p$-times differentiable and satisfies
    \begin{equation}
        \| \nabla^p g(x) - \nabla^p g(y) \| \leq L_p \|x-y\|
    \end{equation}
    where $L_p \leq O_p(k^{3p/2} (\ln k)^{p})$.
\end{restatable}

The proof relies on the fact that $\smax$ is smooth and hence each $f_i$ is smooth. If $h = f_i$ for a particular $i$ in a $\beta$-neighborhood of $x$, then $g = \S[h]$ would also be smooth  (by~\Cref{lem:smoothing}, item 4). If $h$ depends on multiple $f_i$s in a $\beta$-neighborhood of $x$, then we know that at least two softmax's involved in the definitions of the $f_i$s have nearly the same value in the neighborhood of $x$, and by~\Cref{lem:nearlysamesmax} they have nearly the same gradient. This makes $h$ \emph{nearly} smooth, which will allow us to say that $g$ is smooth at $x$ (by~\Cref{lem:smoothing}, item 5).

\begin{proof}[Proof of Lemma~\ref{lem:gsmoothness}]
    Each $f_i$ is an instance of softmax applied to $\aff_V(x)$ plus a constant. Since $\aff_V(x)$ is the vector $x$ transformed by a unitary and then translated, the smoothness and convexity properties of $\smax_{\rho}$ also apply to $f_i$. Hence each $f_i$ is convex, $p$-times differentiable and its $p$th derivatives are $O_p(\rho^{-p})$-Lipschitz (see \Cref{lem:smaxsmoothness}). The function $h$, being a maximum over convex functions, is also convex. By the properties of the smoothing operator $\S$ (\Cref{lem:smoothing}), the function $g= \S[h]$ is also convex.

    Let $x \in \R^n$. Let $j \in [k]$ be the minimum number such that there is a point $y \in B_{\beta}(x)$ for which $h(y) = f_j(y)$. We can rewrite $h$ as follows: $h(z) = f_j(z) + \max_{i>j} (f_i(z) - f_j(z))$. We call $f_j(z)$ the smooth term and $\max_{i>j} (f_i(z) - f_j(z))$ the non-smooth term. We know that $f_j$ has an $O_p(\rho^{-p})$ upper bound on the Lipschitzness of its $p$-th order derivatives. If all points $y \in B_{\beta}(x)$ satisfy $h(y) = f_j(y)$, then the non-smooth term is $0$ and so it does not change the smoothness of $h$. $g = \S[h]$ will maintain this smoothness (see item 4 of \Cref{lem:smoothing}).

    If the non-smooth term is non-zero at some point in $B_{\beta}(x)$, then we wish to show that the non-smooth term has a small Lipschitz constant in $B_{\beta}(x)$. This would imply, via item 5 of \Cref{lem:smoothing}, that the $p$-th order derivative of the smoothing of the non-smooth term with $\S$ would have a small Lipschitz constant. Towards this let $x'$ be any point in $B_{\beta}(x)$. Let $I_{x'}$ be the set $\{i \in [k] | h(x') = f_{i}(x')\}$. The set of subgradients of the non-smooth term at $x'$ is the convex hull of $\{\nabla (f_i - f_j)(x')\}_{i \in I(x')}$. So if we show that for an arbitrary $i \in I(x')$, $\| \nabla (f_i - f_j)(x') \| \leq L$, then we know that the non-smooth part is $L$-Lipschitz at $x'$. If $i=j$, then the gradient is zero. Let us take an $i \neq j$ (since $j$ is the smallest, in fact $i>j$).
    By convexity of the ball and the continuity of $f_i$ and $f_j$, there must be a point $y$ in $B_{\beta}(x)$ for which $h(y) = f_i(y) = f_j(y)$. Note that $x' \in B_{2\beta}(y)$. 
    
    The statement $f_i(y) = f_j(y)$ translates to
    \begin{align}
        \frac{\smax_{\rho}^{\leq i}(\aff_V(y)) - \smax_{\rho}^{\leq j}(\aff_V(y))}{\rho} = (i-j)n^{-\alpha} << 1.
    \end{align}

    Expanding the expression for $\smax_{\rho}^{\leq i}$ and $\smax_{\rho}^{\leq j}$ we get
    \begin{align}
    \frac{\smax_{\rho}^{\leq i}(\aff_V(y)) - \smax_{\rho}^{\leq j}(\aff_V(y))}{\rho} = &\ln\left(\frac{
        \sum_{\ell =1}^{i} \exp\left(
            \frac{\dotp{y}{v_{\ell}} + (k-\ell)\gamma}{\rho}
        \right)
    }{
        \sum_{\ell =1}^{j} \exp\left(
            \frac{\dotp{y}{v_{\ell}} + (k-\ell)\gamma}{\rho}
        \right)
    }\right)\\
        = &
    \ln\left(1 + \frac{
        \sum_{\ell = j+1}^{i} \exp\left(
            \frac{\dotp{y}{v_{\ell}} + (k-\ell)\gamma}{\rho}
        \right)
    }{
        \sum_{\ell =1}^{j} \exp\left(
            \frac{\dotp{y}{v_{\ell}} + (k-\ell)\gamma}{\rho}
        \right)
    }\right)
    \end{align}
    
    Since $\|x'-y\| \leq 2\beta$, we have that $\abs{\dotp{x'}{v} - \dotp{y}{v}} \leq 2\beta$ for any unit vector $v$. Hence
    \begin{equation}
        \frac{\smax_{\rho}^{\leq i}(\aff_V(x')) - \smax_{\rho}^{\leq j}(\aff_V(x'))}{\rho} \leq
        \ln\left(1 + \frac{
            e^{2\beta/\rho} \sum_{\ell = j+1}^{i} \exp\left(
                \frac{\dotp{y}{v_{\ell}} + (k-\ell)\gamma}{\rho}
            \right)
        }{
            e^{-2\beta/\rho} \sum_{\ell =1}^{j} \exp\left(
                \frac{\dotp{y}{v_{\ell}} + (k-\ell)\gamma}{\rho}
            \right)
        }\right).
    \end{equation}

    For all $c>0$, $\ln(1+e^{4\beta/\rho}c) \leq e^{4\beta/\rho} \ln(1+c)$. So we can conclude from the above that
    \begin{equation}
        \frac{\smax_{\rho}^{\leq i}(\aff_V(x')) - \smax_{\rho}^{\leq j}(\aff_V(x'))}{\rho} \leq (i-j)n^{-\alpha} e^{4\beta/\rho}.
    \end{equation}

    Now by~\Cref{lem:nearlysamesmax}, $\|\nabla (f_i - f_j)(x') \| = \| \nabla \smax^{\leq i}_{\rho}(\aff_V(x')) - \nabla \smax^{\leq j}_{\rho}(\aff_V(x')) \| \leq 4(i-j)n^{-\alpha} e^{4\beta/\rho}$.
    
    Hence the non-smooth part of $h$ is $4kn^{-\alpha}\exp(8\beta/\rho)$-Lipschitz in $B_{\beta}(x)$. The $p$th derivatives of $g = \S[h] = \S[f_j] + \S[\max_{i>j} (f_i - f_j)]$ are thus by \Cref{lem:smoothing}, $O_p(\rho^{-p} + n^p\beta^{-p}kn^{-\alpha}\exp(4\beta/\rho))$-Lipschitz. We know that $\alpha = p+1, \beta = \gamma/\ln n$ and $\rho = \gamma/100 \alpha \ln n$, simplifying our bound to $O_p((\ln n/\gamma)^p)$. Furthermore, $k = \floor{(0.1/\gamma)^{2/3}}$ and $\ln n = O(\ln k)$. Hence we can rewrite our upper bound as $O_p\left(k^{3p/2} (\ln k)^{p}\right)$.
\end{proof}

We now see how to prove the query lower bound on optimizing this function class. In order to do so, we need to introduce some intermediate functions. Let $h_i(x) = \max_{j \in [i]} f_i(x)$ and $g_i(x) = \S[h_i](x)$. Let an oracle call to a function $f$ at a point $x$ be denoted by $\Oracle_f(x) = (f(x),\nabla f(x), \nabla^2 f(x), \dots, \nabla^p f(x))$. 
The following results will hold when the set $V$ of orthonormal vectors is chosen uniformly at random (or Haar randomly).
The next two lemmas about these intermediate functions form the backbone of our lower bound.

\begin{restatable}{lemma}{singlesteplemma}
    \label{lem:singlestep}
    Fix any $t \in [0,\dots,k-1]$. Let $V$ be distributed Haar randomly. Conditioned on any fixing of $\{v_i\}_{i \leq t}$, any query $x$ in the unit ball will satisfy $\Oracle_{g}(x) = \Oracle_{g_{t+1}}(x)$ with probability at least $1-1/n^{10}$.
\end{restatable}
\begin{restatable}{lemma}{finaloutputlemma}
    \label{lem:finaloutput}
    Let $V$ be distributed Haar randomly. Conditioned on any fixing of $\{v_i\}_{i \leq k-1}$, any point $x$ in the unit ball will be $\epsilon$-optimal for $g$ with probability at most $1/n^{10}$.
\end{restatable}
The proofs of both of these use the following lemma.
\begin{lemma}\label{lem:ballconcentration}
	Fix any $t \in [0,\dots,k-1]$. Conditioned on any fixing of $\{v_i\}_{i \leq t}$, any query $x$ in the unit ball will, with probability $1-1/n^{10}$, satisfy $\forall i>t\,\,\abs{\dotp{v_i}{x}} \leq 10\sqrt{\frac{\ln n}{n}}$.
\end{lemma}

\begin{proof}
	Note that for $i>t$, $v_i$ is distributed uniformly at random from a unit sphere in $\R^{n-t}$. The following useful concentration statement about random unit vectors follows from~\cite[Lemma 2.2]{Ball97}.
	
	\begin{proposition}\label{prop:concentration}
		Let $x\in B(\vec 0, 1)$. Then for a random unit vector $v$, and all $c>0$,
		\begin{equation}
			\Pr_{v}(|\<x,v\>| \geq c) \leq 2e^{-nc^2/2}.
		\end{equation}
	\end{proposition}
	
	Using~\Cref{prop:concentration} and the fact that $n-t > n/2$ we have that for any $x$ in the unit ball,
	\begin{align*}
		\Pr\left[\abs{\dotp{v_i}{x}} \geq 10\sqrt{\frac{\ln n}{n}}\right] &\leq 2e^{-n/2 \left( 10\sqrt{\frac{\ln n}{n}} \right)^2/2}\\
		&\leq 2e^{-25\ln n} \leq n^{-24}.
	\end{align*}
	Applying a union bound for each of the vectors $v_{t+1}, \dots, v_k$, we have that with probability at least $1-1/n^{23}$, $\forall i>t\,\,\abs{\dotp{v_i}{x}} \leq 10\sqrt{\frac{\ln n}{n}}$. (We use the constant $10$ in the lemma statement only because it is a nicer constant than $23$.)
\end{proof}

\begin{proof}[Proof of \Cref{lem:singlestep}]
	To show that $g(x) = g_{t+1}(x)$, we will show that for all $y \in B_{\beta}(x)$, $h(y) = h_{t+1}(y)$. Let $E_x$ be the event that $x$ satisfies $\forall i>t\,\,\abs{\dotp{v_i}{x}} \leq 10\sqrt{\frac{\ln n}{n}}$. We will show that $E_x \implies g(x) = g_{t+1}(x)$. Hence let us assume $E_x$ holds.
	
	We know that $x$ satisfies $\dotp{v_{j}}{x} - \dotp{v_{t+1}}{x} \leq 20\sqrt{\frac{\ln n}{n}}$ for all $j>t+1$. Hence for any $y \in B_{\beta}(x)$, $\dotp{v_{j}}{y} - \dotp{v_{t+1}}{y} \leq 20\sqrt{\frac{\ln n}{n}} + 2\beta$. To show that $h(y) = h_{t+1}(y)$, it is sufficient to show that $f_{t+1}(y) \geq f_j(y)$ for all $j>t+1$.
	
	Note that $f_{t+1}(y) \geq f_j(y)$ if and only if 
	\begin{align*}
		(j-t-1)n^{-\alpha} & \geq 
		\ln\left(\frac{
			\sum_{\ell =1}^{j} \exp\left(
			\frac{\dotp{y}{v_{\ell}} + (k-\ell)\gamma}{\rho}
			\right)
		}{
			\sum_{\ell =1}^{t+1} \exp\left(
			\frac{\dotp{y}{v_{\ell}} + (k-\ell)\gamma}{\rho}
			\right)
		}\right)\\
		& = 
		\ln\left(1 + \frac{
			\sum_{\ell = t+2}^{j} \exp\left(
			\frac{\dotp{y}{v_{\ell}} + (k-\ell)\gamma}{\rho}
			\right)
		}{
			\sum_{\ell =1}^{t+1} \exp\left(
			\frac{\dotp{y}{v_{\ell}} + (k-\ell)\gamma}{\rho}
			\right)
		}\right)
	\end{align*}
	Since $c \geq \ln(1+c)$, the following statement which we will show is in fact stronger.
	\[ n^{-\alpha} \geq \frac{
		k \max_{t+2 \leq \ell \leq j} \exp\left(
		\frac{\dotp{y}{v_{\ell}} + (k-t-2)\gamma}{\rho}
		\right)
	}{
		\exp\left(
		\frac{\dotp{y}{v_{t+1}} + (k-t-1)\gamma}{\rho}
		\right)
	}.\]
	
	This can be rewritten as $- \rho \alpha \ln n \geq \rho \ln k + \max_{t+2 \leq \ell \leq j} \dotp{y}{v_{\ell}} - \dotp{y}{v_{t+1}} - \gamma$, or $\gamma \geq \rho (\ln k + \alpha \ln n) + \max_{t+2 \leq \ell \leq j} \dotp{y}{v_{\ell}} - \dotp{y}{v_{t+1}}$.
	
	We know this last statement is true because the RHS is at most $\rho(1 + \alpha)\ln n + 20\sqrt{\frac{\ln n}{n}} + 2\beta$ which is smaller than $\gamma$, which is $40\sqrt{\frac{\ln n}{n}}$ (recall that $\beta = \gamma/\ln n$ and $\rho = \gamma/100 \alpha \ln n$).
	
	Since $E_x$ is true with probability $1-1/n^{10}$~\Cref{lem:ballconcentration}, the lemma follows.
\end{proof}

\begin{proof}[Proof of \Cref{lem:finaloutput}]
	Again, let $E_x$ be the event that $x$ satisfies $\forall i>t\,\,\abs{\dotp{v_i}{x}} \leq 10\sqrt{\frac{\ln n}{n}}$. Let us assume $E_x$ holds.
	
	The value of $g(x)$ can be lower bound as follows. Since $\dotp{x}{v_k} \geq -10\sqrt{\frac{\ln n}{n}}$, $h(x) \geq f_k(x) \geq \rho \ln \exp\left(\frac{-10\sqrt{\frac{\ln n}{n}}}{\rho}\right) = -10\sqrt{\frac{\ln n}{n}}$. Since $h$ is $1$-Lipschitz, $g(x) \geq -10\sqrt{\frac{\ln n}{n}} - 2\beta \geq -11\sqrt{\frac{\ln n}{n}}$ (because $\beta < 40/\sqrt{n \ln n}$).
	
	For $x^* = \frac{1}{\sqrt{k}} \sum -v_i$, we know each $f_i(x^*)$ is at most $\rho \ln\left( k \exp\left(\frac{-1/\sqrt{k} + k\gamma}{\rho}\right)\right) + kn^{-\alpha}$. This is at most \[ \rho \ln k - \frac{1}{\sqrt{k}} + k\gamma + kn^{-\alpha}. \]This in turn is at most $-0.8/\sqrt{k}$ since $k\gamma \leq 0.1/\sqrt{k}, n \geq \Omega(k^3),\alpha>1$ and $\rho < 1/k$. So $g(x^*) \leq -0.8/\sqrt{k} + 2\beta < -0.7/\sqrt{k}$.
	
	Since $\sqrt{\frac{\ln n}{n}} \ll 1/\sqrt{k}$, $g(x) > g(x^*) + 0.1/\sqrt{k}$ and so $x$ does not optimize $g$.
	
	Since $E_x$ holds with probability $1-1/n^{10}$, the lemma follows.
\end{proof}

To see how Lemmas~\ref{lem:singlestep} and~\ref{lem:finaloutput} above lead us to our lower bound, fix any $k-1$-query algorithm and consider the following experiment. For each $i$ from $1$ to $k-1$, when the algorithm makes its $i$th query do the following.
\begin{itemize}
	\item Sample $v_i$ from the space orthogonal to the vectors $v_1$ to $v_{i-1}$.
	\item Provide the algorithm the value that $\Oracle_{g_{i}}$ returns on the query. Note that the function $g_i$ depends only on the sampled vectors $v_1$ through $v_i$.
\end{itemize}
It follows that the output of the algorithm is independent of the vector $v_k$ (conditioned on the vectors $v_1$ through $v_{k-1}$). Now we use \Cref{lem:finaloutput} to say that with high probability the output of the algorithm is not $\epsilon$-optimal for $g$. We can now use \Cref{lem:singlestep} along with the hybrid argument to conclude that with high probability the transcript of this query algorithm is the same as the actual transcript (i.e. the transcript had $v_1$ to $v_k$ all been sampled at the beginning and all the queries been to $\Oracle_{g}$). Since the transcripts are the same with high probability, the outputs of the algorithms are also the same with high probability. Hence even when all the queries are to $\Oracle_{g}$, with high probability the output is not $\epsilon$-optimal for $g$. This proof is made formal as the proof of~\Cref{thm:randlb}.

\section{Lower bounds}

We can now establish the randomized lower bound using \Cref{lem:gsmoothness}, \Cref{lem:singlestep}, and \Cref{lem:finaloutput}.

\begin{theorem}\label{thm:randlb}
    Let $\mathcal{A}$ be a randomized query algorithm making at most $k-1$ queries to $\Oracle_{g}$. When $V$ is distributed Haar randomly, the probability that the output of $\mathcal{A}$ is $\epsilon$-optimal is $o(1)$.
\end{theorem}

\begin{proof}
    Let the success probability of $\mathcal{A}$ be $p_\mathrm{succ}$ when $V$ is distributed Haar randomly. We can fix the randomness of $\mathcal{A}$ to get a deterministic algorithm $\mathcal{B}$ with success probability at least $p_\mathrm{succ}$ on the same distribution.

    Let us denote the transcript of $\mathcal{B}$ as $\overline{x} = (x_1, x_2, \dots, x_{k-1}, x_\mathrm{out})$ where $x_i$ is the $i$th query made and $x_\mathrm{out}$ is the output of the algorithm. Note that these are random variables that depend only on $V$. We now create hybrid transcripts $\overline{x}^{(i)}$ for $0 \leq i \leq k-1$. The hybrid transcript $\overline{x}^{(i)} = (x_1^{(i)},\cdots,x_{k-1}^{(i)},x_{\mathrm{out}}^{(i)})$ is defined as the transcript of $\mathcal{B}$ when, for all $j \leq i$, its $j$th oracle call (which is supposed to be to $\Oracle_{g}$) is replaced with an oracle call to $\Oracle_{g_j}$. Note that
    \begin{itemize}
        \item For any $V$, $\overline{x} = \overline{x}^{(0)}$.
        \item $\overline{x}^{(k-1)}$ is a function of $\{v_i\}_{i \leq k-1}$.
        \item For any $V$, if $\Oracle_{g}(x_{i}^{(i-1)}) = \Oracle_{g_i}(x_{i}^{(i)})$ then $\overline{x}^{(i-1)} = \overline{x}^{(i)}$. This is because they have queried the same oracles in their first $i-1$ calls, given the same input in the $i$th call and gotten the same output, and have been querying the same oracles thereafter.
    \end{itemize}

    We start with the observation that
    \begin{align}
        \Pr_V[x_{\mathrm{out}}^{(k-1)} \text{ is }\epsilon\text{-optimal}] &= \E_{v_1,\dots,v_{k-1}} \left[ \Pr_{v_k | v_{<k}}[x_{\mathrm{out}}^{(k-1)} \text{ is }\epsilon\text{-optimal}] \right] \\
        &\leq n^{-10}. \tag{by \Cref{lem:finaloutput}}
    \end{align}

    Next we show that $\Pr_V[x_{\mathrm{out}}^{(k-1)} = x_{\mathrm{out}}] \geq 1-o(1)$ which will complete the proof.

    \begin{align}
        \Pr_V[x_{\mathrm{out}} \neq x_{\mathrm{out}}^{(k-1)}] &\leq \sum_{i \in [k-1]} \Pr_V[x_{\mathrm{out}}^{(i-1)} \neq x_{\mathrm{out}}^{(i)}]\\
        &\leq \sum_{i \in [k-1]} \Pr_V[\overline{x}^{(i-1)} \neq \overline{x}^{(i)}]\\
        &\leq \sum_{i \in [k-1]} \Pr_V[\Oracle_{g}(x_{i}^{(i-1)}) \neq \Oracle_{g_i}(x_{i}^{(i)})]\\
        &\leq \sum_{i \in [k-1]} \E_{v_1,\dots,v_{i-1}}\left[\Pr_{v_i,\dots,v_k | v_{<i}}[\Oracle_{g}(x_{i}^{(i-1)}) \neq \Oracle_{g_i}(x_{i}^{(i)})] \right]\\
        &\leq k n^{-10},
    \end{align}
    since $x_{i}^{(i-1)} = x_{i}^{(i)}$, and using \Cref{lem:singlestep}.
\end{proof}

We now translate the above lower bound to the optimization setting and establish the randomized lower bound in \Cref{thm:main}.

\begin{proof}[Proof of randomized lower bound in \Cref{thm:main}]\ \\
    Our hard function class had $L_p = O_p(k^{3p/2} (\ln k)^{p})$, $R=1$, $\epsilon = 0.1/\sqrt{k}$. Given these parameters, $L_pR^{p+1}/\epsilon = O_p(k^{(3p+1)/2}(\ln k)^p)$ and $\ln(L_pR^{p+1}/\epsilon) = O_p(\ln k)$. Hence
    \begin{align}
        \left(\frac{L_pR^{p+1}}{\epsilon}\right)^{\frac{2}{3p+1}} \left(\ln\frac{L_pR^{p+1}}{\epsilon}\right)^{-2/3} &\leq O_p\left(k (\ln k)^{\frac{2p}{3p+1}} (\ln k)^{-2/3}\right) \leq O_p(k).
    \end{align}
    Since we have a lower bound of $k$, the theorem statement follows.
\end{proof}

We mention here that~\Cref{lem:singlestep,lem:finaloutput} have even more far-reaching consequences. In~\Cref{sec:infhiding}, we note that these lemmas also prove a lower bound of $k$ in (a) the parallel randomized setting where polynomially many non-adaptive queries are allowed in each round and we want to bound the number of rounds and (b) the quantum setting. Hence the best known deterministic algorithm that was recently discovered is also nearly optimal amongst parallel randomized and quantum algorithms.

 	\section{Parallel Randomized and Quantum Lower Bounds}\label{sec:infhiding}

Our randomized lower bound follows from two important properties satisfied by our hard class of functions. We abstract out these properties and define a generic class of hard functions. A class of functions $\mathscr{F}$ is an information-hiding class of functions if there is a sequence of `partially-informed' functions that reveal very little new information, with $\mathscr{F}$ containing the `fully-informed' functions.

An example the reader may want to keep in mind is the following `Guess the numbers' problem. For a sequence of numbers $A = (a_1, \dots, a_m) \in [N]^m$, consider the function $f_{A}$ that takes as input a sequence $B \in [N]^m$ and returns the sequence $A_{\leq i} 0^{m-i}$ where $i \in [m]$ is the maximum number such that $A_{<i} = B_{<i}$. The task is to learn $A$. An example `partially-informed' function would be $f_{A_{\leq i}0^{m-i}}$. This hides information in the sense that if one doesn't know $A_{\le i}$ (say $A$ is chosen uniformly at random), then for most inputs the output of $f_{A}$ would be the same as the output of $f_{A_{\leq i}0^{m-i}}$, and of course, for no input does the output of $f_{A_{\leq i}0^{m-i}}$ reveal any more information than $A_{\leq i}$.

\begin{definition}[Information-Hiding Class of Functions]\label{def:infhidingfunctions}
	Let $\mR = (\mR_1, \dots, \mR_m)$ be a random variable defining a sequence of functions $f_1, \dots, f_m$ in the sense that setting a value of $\mR_{\leq i}$ fixes the function $f_i$. The class of functions $\{f_m\}$ obtained by ranging over the various values of $\mR$ is an $m$-step $(\delta_1,\delta_2)$-information-hiding class of functions (under the distribution $\mR$) if the sequence satisfies the following properties.
	\begin{enumerate}
		\item For all $1 \leq i < m$ and any setting of $\mR_{< i}$,
		\[ \forall x \in \R^n: \Pr_{\mR_{\geq i} | \mR_{< i}} 
		\left( O_{f_m}(x) = O_{f_i}(x) \right) \geq 1-\delta_1 \]
		where $O_{f}(x)$ is the information about $f$ that the model allows us to query at $x$ (for example the function value, gradient and perhaps higher order derivatives if our queries provide them).
		\item For any setting of $\mR_{< m}$,
		\[ \forall x \in \R^n: \Pr_{\mR_{m} | \mR_{< m}} 
		\left( x \text{ is a correct output for }f_m \right) \leq \delta_2. \]
	\end{enumerate}
\end{definition}

As a corollary of \Cref{lem:singlestep,lem:finaloutput}, we see that our class of hard functions were indeed information-hiding functions.

\begin{corollary}\label{thm:gisinfhiding}
	Let $V = (v_1, \dots, v_k)$ be the random variable that is distributed Haar randomly from the possible choices of $k$ orthonormal vectors from $\R^n$. The sequence of functions $g_1,\dots,g_k$ is a $k$-step $(n^{-10},n^{-10})$-information-hiding class of functions when the allowed queries are function values and derivatives up to the $p$th order derivative.
\end{corollary}

We now prove the hardness of information-hiding classes of functions. We start with the setting of parallel randomized algorithms.

\begin{theorem}\label{thm:parallellb}
	Let $\mathscr{F}$ be an $m$-step $(\delta_1,\delta_2)$-information-hiding class of functions under the distribution $\mR$. Then for any parallel query algorithm $\mathcal{A}$ making $K$ queries per round and using less than $m$ rounds, the probability that the algorithm outputs a correct output for $f$ distributed according to $\mR$ is at most $\delta_2 + mK\delta_1$.
\end{theorem}

\begin{proof}
	Let the success probability of $\mathcal{A}$ be $p_\mathrm{succ}$ when $V$ is distributed Haar randomly. We can fix the randomness of $\mathcal{A}$ to get a deterministic algorithm $\mathcal{B}$ with success probability at least $p_\mathrm{succ}$ on the same distribution.
	
	Let us denote the transcript of $\mathcal{B}$ as $T = (S_1, S_2, \dots, S_{m-1}, x_\mathrm{out})$ where $S_i$ is the set of queries made in the $i$th round and $x_\mathrm{out}$ is the output of the algorithm. Note that these are random variables that depend only on $\mR$. We now create hybrid transcripts $T^{(i)}$ for $0 \leq i \leq m-1$. The hybrid transcript $T^{(i)} = (S_1^{(i)},\cdots,S_{m-1}^{(i)},x_{\mathrm{out}}^{(i)})$ is defined as the transcript of $\mathcal{B}$ when, for all $j \leq i$, the oracles calls in round $j$ (which are supposed to be to $\Oracle_{f_m}$) are replaced with oracle calls to $\Oracle_{f_j}$. Note that
	\begin{itemize}
		\item For any $V$, $T = T^{(0)}$.
		\item $T^{(m-1)}$ is a function of $\mR_{\leq m-1}$.
		\item For any $V$, if the answers of $\Oracle_{f_m}$ on $S_{i}^{(i-1)}$ are the same as the answers of $\Oracle_{f_i}$ on $S_{i}^{(i)}$ then $T^{(i-1)} = T^{(i)}$. This is because they have queried the same oracles in their first $i-1$ calls, given the same inputs in the $i$th call and gotten the same output, and have been querying the same oracles thereafter.
	\end{itemize}
	
	We start with the observation that
	\begin{align}
		\Pr_{\mR}[x_{\mathrm{out}}^{(m-1)} \text{ is }\epsilon\text{-optimal}] &= \E_{\mR_{< m}} \left[ \Pr_{\mR_m | \mR_{<m}}[x_{\mathrm{out}}^{(m-1)} \text{ is }\epsilon\text{-optimal}] \right] \\
		&\leq \delta_2. \tag{by property $2$ in \Cref{def:infhidingfunctions}}
	\end{align}
	
	Next we show that $\Pr_{\mR}[x_{\mathrm{out}}^{(m-1)} = x_{\mathrm{out}}] \geq 1-mK\delta_1$ which will complete the proof.
	
	\begin{align}
		\Pr_{\mR}[x_{\mathrm{out}} \neq x_{\mathrm{out}}^{(m-1)}] &\leq \sum_{i \in [m-1]} \Pr_{\mR}[x_{\mathrm{out}}^{(i-1)} \neq x_{\mathrm{out}}^{i}]\\
		&\leq \sum_{i \in [m-1]} \Pr_{\mR}[T^{(i-1)} \neq T^{(i)}]\\
		&\leq \sum_{i \in [m-1]} \Pr_{\mR}[\Oracle_{f_m}(S_{i}^{(i-1)}) \neq \Oracle_{f_i}(S_{i}^{(i)})]\\
		&\leq \sum_{i \in [m-1]} \E_{\mR_{<i}}\left[\Pr_{\mR_{\geq i} | \mR_{<i}}[\Oracle_{f_m}(S_{i}^{(i-1)}) \neq \Oracle_{f_i}(S_{i}^{(i)})] \right]\\
		&\leq m K \delta_1,
	\end{align}
	since $S_{i}^{(i-1)} = S_{i}^{(i)}$, and using property 1 in~\Cref{def:infhidingfunctions} with a union bound over the inputs in each $S_i$.
\end{proof}

We now turn to quantum query algorithms. In our quantum query model a $t$-query quantum query algorithm is a quantum circuit that uses a query oracle $t$ times. The query oracle is implemented by a unitary so that it supports queries in superposition. We allow arbitrarily high precision for the real numbers involved, and our lower bound is independent of the algorithm maker's choice of number of bits of precision. This is the same model used by and described in more detail in~\cite[Section 4.3]{GKNS21}.

We show that an information-hiding class of functions would be hard even for quantum query algorithms to compute. We can't use the above proof since we can't use a union bound on all queried points; a single quantum query may query exponentially many points in superposition. However, we know that a large fraction of this superposition is on points that don't reveal much information. The small fraction of points that do reveal information will not be noticeable to the quantum query algorithm since they are only a small fraction of the superpositioned points. We can then use the hybrid argument again to give a quantum query lower bound analogous to the classical one proved above.

\begin{theorem}\label{thm:quantumlb}
	Let $\mathscr{F}$ be an $m$-step $(\delta_1,\delta_2)$-information-hiding class of functions under the distribution $\mR$. Then for any quantum query algorithm making less than $m$ queries, the probability that the algorithm outputs a correct output for $f$ distributed according to $\mR$ is at most $\delta_2 + 4m\sqrt{\delta_1}$.
\end{theorem}

The proof of this goes via what is commonly called the hybrid argument. Fix any quantum algorithm $A$ making at most $m-1$ queries, specified by the unitaries $U_{m-1}O_{f_m}U_{m-2}O_{f_m} \cdots U_{1}O_{f_m}U_{0}$. Now we define a sequence of unitaries starting with $A_0 = A$ as follows:
\begin{align}\label{eq:unitaries}
	A_0 &\defeq U_{m-1}O_{f_m}U_{m-2}O_{f_m} \cdots O_{f_m}U_{1}O_{f_m}U_{0} \nonumber\\
	A_1 &\defeq U_{m-1}O_{f_m}U_{m-2}O_{f_m} \cdots O_{f_m}U_{1}O_{f_1}U_{0} \nonumber \\
	A_2 &\defeq U_{m-1}O_{f_m}U_{m-2}O_{f_m} \cdots O_{f_2}U_{1}O_{f_1}U_{0}\\
	&\phantom{n}\vdots \nonumber\\
	A_{m-1} &\defeq U_{m-1}O_{f_{m-1}}U_{m-2}O_{f_{m-2}} \cdots O_{f_2}U_{1}O_{f_1}U_{0} \nonumber
\end{align}

Property 1 provides us with the following lemma.

\begin{lemma}[$A_{t}$ and $A_{t-1}$ have similar outputs]\label{lem:similar}
	Let $A$ be a $m-1$ query algorithm and let $A_{t}$ for $t\in [m-1]$ be the unitaries defined in \cref{eq:unitaries}. Then
	\begin{equation}
		\E_{\mR}\bigl(\norm{A_t|0\>-A_{t-1}|0\>}^2\bigr) \leq 4\delta_1.
	\end{equation}
\end{lemma}
\begin{proof}
	From the definition of the unitaries in \cref{eq:unitaries} and the unitary invariance of the spectral norm, we see that $\norm{A_t|0\>-A_{t-1}|0\>} = 
	\norm{ (O_{f_t}-O_{f_m}) U_{t-1}O_{f_{t-1}} \cdots O_{f_1} U_0|0\>}$. Let us prove the claim for any fixed choice of vectors $\mR_{\leq t-1}$, which will imply the claim for any distribution over those vectors. Once we have fixed these vectors, the state $U_{t-1}O_{f_{t-1}} \cdots O_{f_1} U_0|0\>$ is a fixed state, which we can call $|\psi\>$. 
	Thus our problem reduces to showing for all quantum states $|\psi\>$,
	\begin{equation}\label{eq:OVOVt}
		\E_{\mR_{\ge t} | \mR_{< t}}\bigl(\norm{(O_{f_t}-O_{f_m})|\psi\>}^2\bigr) \leq 4\delta_1. 
	\end{equation}
	Now we can write an arbitrary quantum state as $|\psi\>=\sum_x \alpha_x |x\>|\phi_x\>$, where $x$ is the query made to the oracle, and $\sum_x |\alpha_x|^2 =1$.  Thus the LHS of \cref{eq:OVOVt} is equal to
	\begin{equation}
		\E_{\mR_{\geq t} | \mR_{< t}}\left(\Norm{\sum_{x} \alpha_x (O_{f_t}-O_{f_m})|x\>|\phi_x\>}^2\right)
		\leq 
		\sum_{x}  |\alpha_x|^2 \E_{\mR_{\geq t} | \mR_{< t}}\left(\norm{  (O_{f_t}-O_{f_m})|x\>|\phi_x\>}^2\right)        .
	\end{equation}
	
	Since $|\alpha_x|^2$ defines a probability distribution over $x$, we can again upper bound the right hand side for any $x$ instead. 
	Since $O_{f_t}$ and $O_{f_m}$ behave identically for some inputs $x$, the only nonzero terms are those where the oracles respond differently, which can only happen if $O_{f_t}(x) \neq O_{f_m}(x)$. When the response is different, we can upper bound $\norm{  (O_{f_t}-O_{f_m})|x\>|\phi_x\>}^2$ by $4$ using the triangle inequality. Thus for any $x \in \R^n$, we have 
	\begin{equation}
		\E_{\mR_{\geq t} | \mR_{< t}}\left(\norm{(O_{f_t}-O_{f_m})|x\>|\phi_x\>}^2\right)
		\leq 4 \Pr_{\mR_{\geq t} | \mR_{< t}}(O_{f_t}(x) \neq O_{f_m}(x)) \leq 4\delta_1, 
	\end{equation}
	where the last inequality follows from Property $1$.
\end{proof}

And Property 2 provides us with the following.

\begin{lemma}[$A_{m-1}$ does not solve the problem]\label{lem:Akminusone}
	Let $A$ be a $m-1$ query algorithm and let $A_{m-1}$ be as defined in \cref{eq:unitaries}. 
	Let $p_R$ be the probability distribution over $x \in B(\vec 0,1)$ obtained by measuring the output state $A_{m-1}|0\>$ when the randomness $\mR$ is fixed to $R$. Then $\Pr_{R \sim \mR, x \sim p_R} (x \text{ is a correct output}) \leq \delta_2$.
\end{lemma}
\begin{proof}
	Let us establish the claim for any fixed choice of $\mR_{< m}$, since if the claim holds for any fixed choice of these vectors, then it also holds for any probability distribution over them. For a fixed choice of vectors, this claim is just $\Pr_{\mR_m, x \sim p_R} (x \text{ is a correct output}) \leq \delta_2$. Now since the algorithm $A_{m-1}$ only has oracles $O_{f_i}$ for $i<m$, the probability distribution $p_R$ only depends on $R_{< m}$. Since these are fixed, this is just a fixed distribution $p$. So we can instead establish our claim for all $x \in B(\vec 0, 1)$, which will also establish it for any distribution.
	
	So what we need to establish is that for any $x \in \R^n$, $\Pr_{\mR_{m}} \left( x \text{ is a correct output} \right) \leq \delta_2$ which is what Property 2 gives us.
\end{proof}

Finally we can put these two lemmas together to prove our lower bound.

\begin{proof}[Proof of~\Cref{thm:quantumlb}]
	Let $A$ be an $m-1$ query algorithm. Let $p_R$ be the probability distribution over $x \in B(\vec 0,1)$ obtained by measuring the output state $A|0\>$ when the randomness $\mR$ is fixed to $R$. We will show that $\Pr_{R \sim \mR, x\sim p_R} (x \text{ is a correct output}) \leq \delta_2 + 4m\sqrt{\delta_1}$, which proves the theorem.

	To this end, let $P_R$ be the projection operator that projects a quantum state $\ket{\psi}$ onto the space spanned by vectors $\ket{x}$ for $x$ such that $x$ is a correct output when $\mR = R$. Then $\| P_R A \ket{0} \|^2 = \Pr_{x\sim p_R} (x \text{ is a correct output})$. We know from \Cref{lem:Akminusone} that $\E_{R \sim \mR}\bigl(\norm{P_R A_{m-1} \ket{0}}^2\bigr) \leq \delta_2$. We prove our upper bound on the probability by showing that it is approximately the same as $\E_{R \sim \mR}\bigl(\norm{P_R A_{m-1} \ket{0}}^2\bigr)$.
	
	\Cref{lem:similar} states that for all $1 \leq t < m$, $\E_{\mR}\bigl(\norm{A_{t}|0\>-A_{t-1}|0\>}^2\bigr) \leq 4\delta_1$. Using telescoping sums and the Cauchy-Schwarz inequality, we see that
	\begin{align}
		\E_{\mR}\bigl( \norm{A_{m-1}|0\>-A|0\>}^2 \bigr) &\leq \E_{\mR}\left( \left(\sum_{t \in [m-1]} \norm{A_{t}|0\>-A_{t-1}|0\>}\right)^2 \right)\\
		&\leq \E_{\mR}\left(\sum_{t \in [m-1]} \norm{A_{t}|0\>-A_{t-1}|0\>}^2\right) \left( \sum_{t \in [m-1]} 1^2 \right) \leq 4\delta_1 \cdot m \cdot m.
	\end{align}
	
	For all $R$, $\abs{\norm{P_R A_{m-1} \ket{0}} - \norm{P_R A \ket{0}}} \leq \norm{P_R A_{m-1} \ket{0} - P_R A \ket{0}} = \norm{P_R (A_{m-1} \ket{0} - A \ket{0})}  \leq \norm{A_{m-1} \ket{0} - A \ket{0}}$. Hence
	\begin{equation}\label{eq:qlowerboundintermediatebound}
		\E_{R \sim \mR}\bigl(\bigl(\norm{P_R A_{m-1} \ket{0}} - \norm{P_R A \ket{0}}\bigr)^2\bigr) \leq 4m^2\delta_1.
	\end{equation}
	
	We want an upper bound on $\E_{R \sim \mR}\bigl(\norm{P_R A \ket{0}}^2 - \norm{P_R A_{m-1} \ket{0}}^2\bigr)$, which is no larger than $2 \E_{R \sim \mR}\bigl(\norm{P_R A \ket{0}} - \norm{P_R A_{m-1} \ket{0}}\bigr)$ since $\norm{P_R A \ket{0}} + \norm{P_R A_{m-1} \ket{0}} \leq 2$. We get such a bound by applying Jensen's inequality to \cref{eq:qlowerboundintermediatebound}: $\E_{R \sim \mR}\bigl(\norm{P_R A \ket{0}} - \norm{P_R A_{m-1} \ket{0}}\bigr) \leq 2m\sqrt{\delta_1}$, and so $\E_{R \sim \mR}\bigl(\norm{P_R A \ket{0}}^2 - \norm{P_R A_{m-1} \ket{0}}^2\bigr) \leq 4m\sqrt{\delta_1}$.
	
	We can now use linearity of expectation and upper bound our required probability as 
	\begin{equation} 
		\Pr_{R \sim \mR,x\sim p_R} (x \text{ is a correct output}) = \E_{R \sim \mR}\bigl(\norm{P_R A \ket{0}}^2\bigr) \leq \delta_2 + 4m\sqrt{\delta_1}. \qedhere
	\end{equation}
\end{proof}

The proofs of the quantum lower bound in \Cref{thm:main} and the highly parallel lower bound alluded to after that now follow from \Cref{thm:parallellb,thm:quantumlb}~and~\Cref{thm:gisinfhiding}.

\begin{corollary}
	The complexity of $\epsilon$-optimizing the class of functions $g$ is:
	\begin{itemize}
		\item $k$ rounds in the parallel randomized setting where in each round $K$ parallel queries are allowed, and $Kn^{-9}\ll 1$. (Note that by modifying the constants in the definition of the function, we can support $K$ being any polynomial in $n$.)
		\item $k$ queries in the quantum setting to get success probability larger than $n^{-4}$.
	\end{itemize}
\end{corollary}
 	\section{Conclusion}
In this paper, we obtained near optimal oracle lower bounds for $p^{\textrm{th}}$-order smooth convex optimization, for any constant $p$, for randomized algorithms. Our results further hold for quantum algorithms and parallel randomized algorithms as well. To obtain our results, we introduce a new smoothing operator that could be of independent interest. An interesting open problem is to obtain both tight upper and lower bounds when we have oracle access only up to $q^\textrm{th}$-order derivatives even though the function is guaranteed to have Lipschitz $p^\textrm{th}$-order derivatives, for $q < p$.
 	

\newcommand{\etalchar}[1]{$^{#1}$}

\appendix

\section{Deferred proofs: Smoothing preliminaries}\label{sec:prelims}

Recall the definitions of our smoothing operators.

Let $B_{\eta}(x)$ be the ball of radius $\eta$ around $x$. For any function $f:\R \rightarrow \R$ and real-valued $\eta > 0$, define the function $S_{\eta}[f]:\R^n \rightarrow \R$ as $S_{\eta}[f](x) = \E_{y \in B_{\eta}(x)}[f(y)]$. 

\smoothing*

We claimed the following about $\S$.

\smoothingproperties*

While most of the above statements are proven below from first principles, a couple of them follow from~\cite[Corollary~2.4]{agarwal2018lower}. For these statements, we only detail the main ideas involved in their proofs.
\begin{proof}
    We prove the above statements in order.
    \begin{enumerate}
        \item This is a simple consequence of the linearity of expectation.
        \item By expanding the expectations in the definition of $\S$, we get that $\S[f](x) = \E_{y \sim \mu_x}[f](x)$ where $\mu_x$ is a distribution supported in $B_{(1-2^{-p})\beta}(x)$.
        \item The gradient and higher order derivatives of $\S[f]$ at $x$ depend only on the values of $\S[f]$ in an open ball around $x$, say $B_{2^{-p}\beta}(x)$. For any $y \in B_{2^{-p}\beta}(x)$, $\S[f](y)$ depends only on values of $f$ in $B_{(1-2^{-p})\beta}(y) \subseteq B_{\beta}(x)$.
        \item This follows from the proof of~\cite[Corollary~2.4]{agarwal2018lower}. It is easy to see that if $f$ is $L$-Lipschitz then $S_{\eta}[f]$ is also $L$-Lipschitz. $\nabla f$ being $L$-Lipschitz is equivalent to saying that for any unit vector $v \in \R^n$, $g_v(x) \defeq \nabla f(x) [v]$ is an $L$-Lipschitz function. However by linearity of $S_{\eta}$, $S_{\eta}[g_v](x) = \nabla S_{\eta}[f](x) [v]$. Hence $\nabla S_{\eta}[f]$ is $L$-Lipschitz. A repeated usage of this argument as done in~\cite{agarwal2018lower} proves the statement.
        \item The proof of this is via a repeated usage of~\cite[Lemma~2.3]{agarwal2018lower} as done in~\cite[Corollary~2.4]{agarwal2018lower}. They argue via Stoke's theorem that $S_{\eta}[f]$ is differentiable even when $f$ may not be differentiable in a set of measure $0$, and that $\nabla S_{\eta}[f]$ is $\frac{n}{\eta}G$-Lipschitz. They then use directional derivatives to inductively show via the same argument the Lipschitzness of the higher-order derivatives.
        \item This is a simple consequence of the fact that $\S[f](x)$ is a convex combination of the values $f(y)$ for $y \in B_{\beta}(x)$.
        \item For $y \in B_{\beta}(\vec{0})$, let $f_y$ be defined as $f_y(x) = f(x+y)$. Since $f$ is convex, it follows that each $f_y$ is convex. Also $\S[f] = \E_{y \in B_{\beta}(\vec{0})}[f_y]$. Since $\S[f]$ is a convex combination of convex functions, it follows that $\S[f]$ is convex.\qedhere
    \end{enumerate}
\end{proof}

Recall the definition of the softmax function.

\softmax*

We claimed the following about the softmax function.

\softmaxproperties*

\begin{proof}
    We prove the statements in order.
    \begin{enumerate}
        \item This is straightforward.
        \item We can conclude the $1$-Lipschitzness by looking at the norm of the gradients.\\
        \begin{equation}
            \| \nabla \smax_{\rho}(x) \| = \frac{\| (\exp(x_1/\rho),\dots,\exp(x_n/\rho)) \|}{\sum_i \exp(x_i/\rho)} \leq \frac{\sum_i \abs{\exp(x_i/\rho)}}{\sum_i \exp(x_i/\rho)} = 1.
        \end{equation}
        The convexity follows from analyzing the Hessian. We get the following as the Hessian.\\
        \begin{equation}
            \nabla^2 \smax_{\rho}(x)_{i,j} =
            \begin{dcases}
                \frac{1}{\rho} \left( \frac{\exp(x_i/\rho)}{\sum_{t \in [n]}\exp(x_t/\rho)} - \frac{\exp(2x_i/\rho)}{(\sum_{t \in [n]}\exp(x_t/\rho))^2} \right) &\text{ if $i=j$}\\
                \frac{1}{\rho} \left( - \frac{\exp((x_i + x_j)/\rho)}{(\sum_{t \in [n]}\exp(x_t/\rho))^2} \right) &\text{ otherwise}
            \end{dcases}
        \end{equation}
        Let $v \in \R^n$ denote the column vector $\nabla \smax_{\rho}(x)$. Then $\nabla^2 \smax_{\rho}(x) = \frac{1}{\rho} (\mathrm{diag}(v) - vv^{\mathsf{T}})$. The convexity of $\smax_{\rho}$ is equivalent to $\nabla^2 \smax_{\rho}(x)$ being positive semidefinite for all $x$. Since $\rho>0$, it suffices to prove that $M = \rho \nabla^2 \smax_{\rho}(x)$ is positive semidefinite. To this end, let $y \in \R^n$ be any column vector.\\
        \begin{align}
            y^{\mathsf{T}} M y &= \sum_{i \in [n]} y_i^2 v_i - \dotp{y}{v}^2\\
            &= \left(\sum_{i \in [n]} y_i^2 v_i\right) \left(\sum_{i \in [n]} v_i\right) - \left(\sum_{i \in [n]} y_iv_i\right)^2 \tag{since $\sum_{i \in [n]} v_i = 1$}\\
            &\geq 0.
        \end{align}
        The last inequality follows by using the Cauchy-Schwarz inequality on the vectors $(y_i \sqrt{v_i})_{i \in [n]}$ and $(\sqrt{v_i})_{i \in [n]}$. Here we use the fact that each $v_i$ is nonnegative.
        \item This is proven in~\cite[Theorem~7]{bullins20highlysmooth}.\qedhere
    \end{enumerate}
\end{proof}

\softmaxgradients*

\begin{proof}
    $\frac{\smax_{\rho}(x) - \smax^{\leq m}_{\rho}(x)}{\rho} = \delta$ implies that
    \begin{equation}
    \delta = \ln\left(\frac{
        \sum_{i=1}^{n} \exp\left(
            \frac{x_i}{\rho}
        \right)
    }{
        \sum_{i=1}^{m} \exp\left(
            \frac{x_i}{\rho}
        \right)
    }\right)
        =
    \ln\left(1 + \frac{
        \sum_{i= m+1}^{n} \exp\left(
            \frac{x_i}{\rho}
        \right)
    }{
        \sum_{i=1}^{m} \exp\left(
            \frac{x_i}{\rho}
        \right)
    }\right)
    \end{equation}
    
    Let $c = \frac{\sum_{i=m+1}^{n} \exp(x_i/\rho)}{\sum_{i=1}^{m} \exp(x_i/\rho)}$. Since $\delta = \ln(1+c) \geq c/2$ for $\delta<1$, an upper bound of $2c$ would suffice to prove the lemma.

    Using the equation for the gradient of $\smax$ from~\Cref{lem:smaxsmoothness}, we get $\| \nabla \smax_{\rho}(x) - \nabla \smax^{\leq m}_{\rho}(x) \|$ to be equal to
    \begin{align}
        &\frac{(\exp\left(\frac{x_1}{\rho}\right), \dots, \exp\left(\frac{x_{n}}{\rho}\right))}{ \sum_{i=1}^{n} \exp\left(\frac{x_{i}}{\rho}\right)} - \frac{(\exp\left(\frac{x_{1}}{\rho}\right), \dots, \exp\left(\frac{x_{m}}{\rho}\right), 0, \dots, 0 )}{ \sum_{i=1}^{m} \exp\left(\frac{x_{i}}{\rho}\right)}\\
        = &\frac{(\exp\left(\frac{x_{1}}{\rho}\right), \dots, \exp\left(\frac{x_{n}}{\rho}\right))}{ \sum_{i=1}^{n} \exp\left(\frac{x_{n}}{\rho}\right)} - \frac{(1+c)(\exp\left(\frac{x_{1}}{\rho}\right), \dots, \exp\left(\frac{x_{m}}{\rho}\right), 0, \dots, 0 )}{ (1+c)\sum_{i =1}^{m} \exp\left(\frac{x_{i}}{\rho}\right)}\\
        = &\frac{-c(\exp\left(\frac{x_{1}}{\rho}\right), \dots, \exp\left(\frac{x_{m}}{\rho}\right)) + (0,\cdots,0,\exp\left(\frac{x_{m+1}}{\rho}\right), \dots, \exp\left(\frac{x_{n}}{\rho}\right), 0, \dots, 0 )}{ (1+c)\sum_{i =1}^{m} \exp\left(\frac{x_{i}}{\rho}\right)}
    \end{align}
    The norm of this is at most
    \begin{align}
        & \frac{\|-c(\exp\left(\frac{x_{1}}{\rho}\right), \dots, \exp\left(\frac{x_{m}}{\rho}\right))\| + \|(\exp\left(\frac{x_{m+1}}{\rho}\right), \dots, \exp\left(\frac{x_{n}}{\rho}\right))\|}{ (1+c)\sum_{i =1}^{m} \exp\left(\frac{x_{i}}{\rho}\right)}\\
        \leq & \frac{c}{1+c} \frac{\sum_{i =1}^{m} \abs{\exp\left(\frac{x_{i}}{\rho}\right)}}{\sum_{i =1}^{m} \exp\left(\frac{x_{i}}{\rho}\right)} + \frac{1}{1+c} \frac{\sum_{i =m+1}^{n} \abs{\exp\left(\frac{x_{i}}{\rho}\right)}}{\sum_{i =1}^{m} \exp\left(\frac{x_{i}}{\rho}\right)}\\
        \leq & \frac{c}{1+c} + \frac{c}{1+c} < 2c.\qedhere
    \end{align}
\end{proof}

\end{document}